\documentclass[reqno,12pt,a4letter]{amsart}
\usepackage{amsmath, amsxtra, amssymb, latexsym, amscd, amsthm}
\usepackage{graphicx, color}
\usepackage[active]{srcltx}
\usepackage[utf8]{inputenc}
\usepackage[mathscr]{euscript}
\usepackage{mathrsfs,cite}
\usepackage[english]{babel}
\usepackage{color,enumerate,comment}
\usepackage[active]{srcltx}
\def\disp{\displaystyle}
\newtheorem{theorem}{Theorem}[section]

\newtheorem{lemma}[theorem]{Lemma}

\theoremstyle{definition}
\newtheorem{definition}[theorem]{Definition}
\newtheorem{example}[theorem]{Example}
\newtheorem{remark}[theorem]{Remark}

\newcommand{\Limsup}{\mathop{{\rm Lim}\,{\rm sup}}}
\numberwithin{equation}{section}

\title[On AKKT optimality conditions for cone-constrained CCVPs]{On AKKT optimality conditions for cone-constrained vector optimization problems}

\author[N.V. Tuyen]{Nguyen Van Tuyen}
\address[N.V. Tuyen]{School of Mathematical Sciences, University of Electronic Science and Technology of China, Chengdu, P.R. China; Department of Mathematics, Hanoi Pedagogical University 2, Xuan Hoa, Phuc Yen, Vinh Phuc, Vietnam}
\email{\tt tuyensp2@yahoo.com; nguyenvantuyen83@hpu2.edu.vn}

\author[Y.-B. Xiao]{Yi-Bin Xiao}
\address[Y.-B. Xiao]{School of Mathematical Sciences, University of Electronic Science and Technology of China, Chengdu, P.R. China}
\email{\tt xiaoyb9999@hotmail.com}

\author[T.Q. Son]{Ta Quang Son$^*$}

\address[T.Q. Son]{Faculty of Applied Mathematics and Applications, Saigon University,  Hochiminh City, Vietnam}
\email{\tt taquangson@sgu.edu.vn}


\thanks{$^*$Corresponding.}
\thanks{The research of Nguyen Van Tuyen was supported by the Ministry of Education and Training of Vietnam [grant number B2018-SP2-14]  as well as the grant from School of Mathematical Sciences, University of Electronic Science and Technology of China, Chengdu, P.R. China. The research of Yi-bin Xiao  was supported by the National Natural Science Foundation of China (11771067).  The research of Ta Quang Son was supported by Vietnam National Foundation for Science and Technology Development (NAFOSTED) under grant number 101.01-2017.08.}
\date{\today}

\begin{document}

\begin{abstract}
In this paper, we introduce a kind of approximate Karush--Kuhn--Tucker condition (AKKT) for a smooth cone-constrained vector optimization problem. We show that, without any constraint qualification, the AKKT condition is a necessary for a local weak efficient solution of the considered problem. For convex problems, we prove that the AKKT condition is a necessary and sufficient optimality condition for a global weak efficient solution. We also introduce some strict  constraint qualifications associated with the AKKT condition.
\end{abstract}

\maketitle


\section{Introduction} \label{introduction}
In this paper, we are concerned with the following cone-constrained vector optimization problem:
 \begin{equation}\label{problem}
 \text{\rm Min}_{\,\mathbb{R}^m_+}\,\big\{f(x)\,:\,x\in \mathbb{R}^n,\; g(x)\in -\Theta \big\},\tag{CCVP}
 \end{equation}
 where $f$ is a mapping from $\mathbb{R}^n$ to $\mathbb{R}^m$, $g$ is a mapping from $\mathbb{R}^n$ to a Banach space $Y$, and $\Theta\subset Y$ is a closed  convex cone. When $m=1$, then \eqref{problem} is called a cone-constrained optimization problem. Since $\Theta$ is in an arbitrary Banach space $Y$, the problems of the form \eqref{problem} contain a wide class of problems in mathematical programming such as semi-infinite optimization problems, semidefinite optimization problems, second-order cone programs, and so on.   

It is well known that Karush--Kuhn--Tucker (KKT)  optimality conditions play an important role in both the theory and practice of constrained optimization and  are very relative to the class of important problems in mechanics and engineering, i.e., variational inequality problems, see, e.g., \cite{Facchinei2003,Kinderlehrer200,HXHW,LXH,SHX,SX2,XS3}. In order to obtain optimality conditions of KKT-type, constraint qualifications are
indispensable; see, e.g.,  \cite{Andreani2011,Bertsekas99,Giorgi16,Nocedal99,Tuyen18}.   A  constraint qualification (CQ) of \eqref{problem} is a condition that ensures that every  optimal solution of \eqref{problem} satisfies the KKT condition. In other words, usual necessary optimality conditions are of the form KKT or not-CQ. In the case of without CQ, sequential optimality conditions (or approximate  KKT optimality conditions) are considered. Recently, sequential optimality conditions  have been recognized to be useful in designing algorithms for finding approximate optimal solutions of nonlinear programming problems; see, e.g.,  \cite{Andreani2007,Andreani2010,Andreani16,Andreani-2018,Andreani-2018-2,Birgin2014,Chuong16,Dutta2010,Dutta2013,Feng2018,Haeser11,Haeser2015,Li15,Long18,Martinez2003}.

For convex cone-constrained optimization problems, there have been some
papers  in the literature dealing with sequential optimality conditions; see, e.g., \cite{Bai2008,Bot2008,Bot2008-2,Dinh2010,Jeyakumar-2003,Jeyakumar2006,Lee2007}. However,  to the best of our knowledge, there has been only one work \cite{Mordukhovich2014} concerning  sequential optimality conditions for nonconvex cone-constrained optimization problems.  Based on advanced constructions and techniques of variational analysis and generalized differentiation,  Mordukhovich {\em et al.} \cite{Mordukhovich2014} derived new necessary optimality conditions in fuzzy  form  for nonsmooth and nonconvex cone-constrained optimization problems. As a consequence of these fuzzy optimality conditions, the authors obtained optimality conditions of approximate KKT-type  without any constraint qualifications. 

In this paper, motivated and inspired by the work of Mordukhovich {\em et al.}  \cite{Mordukhovich2014}, we introduce  an approximate KKT condition  for cone-constrained vector optimization \eqref{problem}. We show that the approximate KKT condition is a necessary one for local weak efficient solutions. Under convexity assumptions, we  prove that this condition is also a sufficient optimality condition. We also introduce some strict  constraint qualifications under which the approximate KKT condition implies the KKT condition.

The paper is organized as follows. In Section \ref{Section2}, we recall some basic definitions and preliminaries from variational analysis and generalized differentiation. Section \ref{Section3} is devoted to presenting the main results.


\section{Preliminaries}\label{Section2}
We use the following notation and terminology. Fix $n \in {\mathbb{N}}:=\{1, 2, \ldots\}$. The space $\mathbb{R}^n$ is equipped with the usual scalar product and  Euclidean norm. 

For a Banach space $X$, the bracket $\langle \cdot\,, \cdot\rangle$ stands for the canonical pairing between space $X$ and its dual $X^*$. The weak$^\ast$ convergence in $X^\ast$ is denoted by  $\stackrel{w^\ast}\longrightarrow$. The topological closure, the topological interior  and the conic hull of a subset  $\Omega$ of $X$ are denoted, respectively, by  $\mathrm{cl}\,{\Omega}$, $\mathrm{int}\,{\Omega}$  and $\mathrm{cone}\,\Omega$. The closed ball with center $x$ and radius $\delta$ is denoted by $B(x, \delta)$.

Here, we recall the notions of the normal cones to  nonempty sets and the subdifferential of real-valued functions used in this paper. The reader is referred to \cite{mor06,Mordukhovich2018} for more details.  

\begin{definition}[{see \cite{mor06}}]{\rm Let $\Omega\subset\mathbb{R}^n$ and  $\bar x\in  \mbox{cl}\,\Omega$. The set
		\begin{equation*}
		N(\bar x; \Omega):=\{z^*\in \mathbb{R}^n\;:\;\exists
		x^k\stackrel{\Omega}\longrightarrow \bar x, z^*_k\to z^*,
		z^*_k\in {\widehat N}(x^k; \Omega),\ \ \forall k \in\mathbb{N}\},
		\end{equation*}
is called the {\em Mordukhovich/limiting normal cone}  of $\Omega$ at $\bar x$, where
		\begin{equation*}
		\widehat N (x; \Omega):= \bigg\{ {z^*  \in {\mathbb{R}^n} \;:\;\limsup_{u\overset{\Omega} \rightarrow x}
			\frac{{\langle z^* , u - x\rangle }}{{\parallel u - x\parallel }} \leq 0} \bigg\}
		\end{equation*}
		is the set of  {\em regular/Fr\'echet normals} of $\Omega$  at $x$ and  $u\xrightarrow {{\Omega}} x$ means that $u \rightarrow x$ and $u \in \Omega$.
	}
\end{definition}

\begin{definition}[{see \cite{mor06}}]{\rm Consider a function $\varphi \colon \mathbb{R}^n \to  \mathbb{R}$ and a point $\bar{x}\in\mathbb{R}^n$. The {\it Mordukhovich/limiting subdifferential}  of $\varphi$ at $\bar x$ is defined by
\begin{align*}
\partial \varphi (\bar x):=\{x^*\in \mathbb{R}^n \,:\, (x^*, -1)\in N((\bar x, \varphi (\bar x)); \mbox{epi}\varphi )\},
\end{align*}
where $\mbox{epi}\varphi$ is the epigraph of $\varphi$ and is given by
\begin{equation*}
\mbox{epi}\varphi:=\{(x, \alpha)\in\mathbb{R}^n\times\mathbb{R} \;:\;  \varphi(x)\leq \alpha\}.
\end{equation*}
}
\end{definition}
The following lemma will  be used in the next section.
\begin{lemma}[{see \cite[Theorem~3.46]{mor06}}]\label{Max-rule}
Let $\varphi_1,\ldots,\varphi_m\colon{\mathbb R}^n\rightarrow{\mathbb R}$ be continuously differentiable functions and $\bar{x}\in\mathbb{R}^n$. For each $x\in\mathbb{R}^n$, put 
$$\varphi(x):=\max_{1\le i\le m}\varphi_i(x).$$ 
Then, we have  
\begin{eqnarray*}
\partial\varphi(\bar{x})&\subset& \left\{\disp\sum_{i\in I(\bar{x})}\alpha_i\nabla\varphi_i(\bar{x})\;:\;\alpha_i\ge 0\quad\textrm{ and }\quad\disp\sum_{i\in I(\bar{x})}\alpha_i=1\right\},
\end{eqnarray*}
where   $I(\bar{x}):=\{i\in \{1, \ldots, m\}\,:\,\varphi_i(\bar{x})=\varphi(\bar{x})\}$.
\end{lemma}

\section{Main results}\label{Section3}
Hereafter  we assume that $f\colon \mathbb{R}^n\to \mathbb{R}^m$ and $g\colon\mathbb{R}^n\to Y$ are continuously Fr\'echet differentiable  on $\mathbb{R}^n$. The derivatives of $f$ and $g$ at $x\in\mathbb{R}^n$ are denoted by $\nabla f(x)$ and $\nabla g(x)$, respectively. The adjoint operator of $\nabla g(x)$ is denoted by $\nabla g(x)^\ast$.    The  polar cone  of $\Theta$ is denoted by
\begin{equation*}
\Theta_+:=\{\mu\in Y^*\;:\; \langle \mu, \theta\rangle\geq 0, \ \ \forall \theta\in\Theta\}.
\end{equation*} 
Clearly, $\Theta_+$ is a weak$^\ast$ closed subset of $Y^*$. For convenience, let $\mathcal{F}$ denote the feasible set of \eqref{problem}, i.e.,
\begin{equation*}
\mathcal{F}:=\{x\in\mathbb{R}^n\;:\; g(x)\in -\Theta\}.
\end{equation*}

A point  $\bar x\in \mathcal{F}$ is called a {\em local weak efficient solution} of \eqref{problem} if and only if  there exists a neighborhood $U$  of $\bar x$ such that there is no $x\in U\cap \mathcal{F}$ satisfying   
$$f_i(x)<f_i(\bar x), \ \ \forall i=1, \ldots, m.$$
If $U=\mathbb{R}^n$, then we say that $\bar x$ is a {\em (global) weak efficient solution} of \eqref{problem}.

\begin{definition}\label{def-KKT}
{\rm We say that $\bar x\in\mathcal{F}$ satisfies the {\em KKT  condition} if and only if there exists a multiplier $(\lambda, \mu)\in\mathbb{R}^m_+\times\Theta_+$ such that
\begin{enumerate}[(i)]
\item $\sum_{i=1}^{m}\lambda_i=1$,

\item $\nabla g(\bar x)^*\mu+\sum_{i=1}^m\lambda_i\nabla f_i(\bar x)=0$,

\item $\langle \mu, g(\bar x)\rangle=0$.
\end{enumerate} 

}
\end{definition}

We now introduce the concept of approximate Karush--Kuhn--Tucker condition for
\eqref{problem} inspired by the work of Mordukhovich {\em et al.} \cite{Mordukhovich2014}. 
\begin{definition}
{\rm We say that $\bar x\in\mathcal{F}$ satisfies the {\em  approximate Karush--Kuhn--Tucker condition} (AKKT) if and only if there exist a vector $\lambda\in \mathbb{R}^m_+$ with $\sum_{i=1}^{m}\lambda_i=1$ and sequences $\{x^k\}\subset \mathbb{R}^n$ and $\{\mu^k\}\subset \Theta_+$ such that  
\begin{enumerate}
\item [(A0)] $x^k\to\bar x$,

\item [(A1)] $\nabla g(x^k)^* \mu^k+\sum_{i=1}^m\lambda_i\nabla f_i(\bar x)\to 0$,

\item [(A2)] $\langle \mu^k, g(x^k)\rangle\to 0$.
\end{enumerate}
}
\end{definition}
A sequence $\{x^k\}$ satisfying the above definition  will be called an AKKT sequence. We note here that the sequence of points $\{x^k\}$ is not required to be feasible.

The following result shows that the AKKT condition is  necessary for a feasible point to be a local weak efficient solution of \eqref{problem} without any constraint qualification.
\begin{theorem}\label{AKKT-Theorem}
If $\bar x$ is a local weak efficient solution of \eqref{problem}, then $\bar x$ satisfies the AKKT condition.
\end{theorem}
\begin{proof} By assumption, there exists $\delta>0$ such that, for all $x\in \mathcal{F}\cap B(\bar x, \delta)$, there exists $i\in \{1, \ldots, m\}$ satisfying
$f_i(x)\geq f_i(\bar x).$ 
This implies that
\begin{equation*}
\psi(x):=\max \{f_i(x)-f_i(\bar x)\;:\; i=1, \ldots, m\}\geq 0,\ \ \ \forall x\in \mathcal{F}\cap B(\bar x, \delta).
\end{equation*}
Consequently, $\bar x$ is a local minimum of the function  $\psi$ on $\mathcal{F}$. Clearly, $\psi$ is Lipschitz continuous around $\bar x$. Hence, by \cite[Corollary 5.1]{Mordukhovich2014}, there exist a subgradient $x^*\in\partial \psi(\bar x)$ and sequences $\{x^k\}\subset \mathbb{R}^n$ and $\{\mu^k\}\subset \Theta_+$ such that
\begin{equation*}
x^k\to \bar x, \nabla g(x^k)^*\mu^k+x^* \to 0, \ \ \text{and}\ \ \langle \mu^k, g(x^k)\rangle \to 0 \ \ \text{as} \ \ k \to \infty.
\end{equation*}
By Lemma \ref{Max-rule}, we have
\begin{equation*}
\partial \psi (\bar x) \subset \left\{\disp\sum_{i=1}^m\lambda_i\nabla f_i(\bar{x})\;:\;\lambda_i\ge 0\ \ \textrm{ and }\ \ \disp\sum_{i=1}^m\lambda_i=1\right\}.
\end{equation*}
Hence  there exists $\lambda\in \mathbb{R}^m_+$ with $\sum_{i=1}^{m}\lambda_i=1$ such that 
$$x^*=\sum_{i=1}^m\lambda_i\nabla f_i(\bar{x}),$$
as required.   
\end{proof}
In the following example, we consider a local weak efficient solution that does not satisfy the KKT condition. Let us then construct an AKKT sequence guaranteed to exist by  Theorem  \ref{AKKT-Theorem}.
\begin{example}
{\rm   
Consider the following problem:
\begin{align*}
	& \text{min}\, f(x):=(f_1(x), f_2(x))
	\\
	&\text{s.t.}\ \ x\in \mathcal{F}:=\{x\in\mathbb{R}^2\,:\, g(x)\in -\Theta\},
\end{align*} 
where $\Theta:=\mathbb{R}^3_+$ and
\begin{align*}
	f_1(x_1, x_2)&:= -3x_1-2x_2+3, f_2(x_1, x_2):=-x_1-3x_2+1,
	\\
	g(x_1, x_2)&:=(g_1(x_1, x_2), g_2(x_1, x_2), g_3(x_1, x_2)),
	\\
	g_1(x_1, x_2)&:=-x_1, g_2(x_1, x_2):=-x_2, g_3(x_1, x_2):=(x_1-1)^3+x_2.
\end{align*}
Let $\bar{x}=(1, 0)$. It is easy to check that $\bar{x}$ is a weak efficient solution, $\Theta_+=\Theta$ and
\begin{align*}
	\nabla f_1(\bar{x})&=(-3, -2), \nabla f_2(\bar{x})=(-1, -3),
	\\
	\nabla g_1(\bar{x})&=(-1, 0), \nabla g_2(\bar{x})=(0, -1), \nabla g_3(\bar{x})=(0, 1).
\end{align*}
Since
\begin{equation*}
	\begin{cases}
	\lambda_1\nabla f_1(\bar{x})+\lambda_2\nabla f_2(\bar{x})+\mu_2\nabla g_2(\bar{x})+\mu_3\nabla g_3(\bar{x})=0
	\\
	\mu_1g_1(\bar x)+\mu_2g_2(\bar x)+\mu_3 g_3(\bar x)=0
	\\
	\lambda_1, \lambda_2, \mu_2, \mu_3 \geq 0
	\end{cases}
\Leftrightarrow
	\begin{cases}
	\lambda_1=\lambda_2=0
	\\
	\mu_1=0, \mu_2=\mu_3 
	\\
	\mu_2, \mu_3\geq 0,
	\end{cases}
\end{equation*}
it follows that there is no multiplier  $(\lambda, \mu)\in \mathbb{R}_+^2\times\Theta_+$ satisfying conditions (i)--(iii) of Definition \ref{def-KKT}. This means that the KKT condition does not hold at $\bar x$.

We now check that the AKKT condition holds at $\bar x$. Let $\lambda=(\frac{1}{2}, \frac{1}{2})$.   We claim that there exist sequences $\{x^k\}\subset \mathbb{R}^2$ and  $\{\mu^k\}\subset \mathbb{R}^3_+$ satisfying conditions (A0)--(A2) with respect to $\lambda$. Indeed, let $x^k=(1+\frac{1}{k}, 0)$ and $\mu^k_1=0$ for all $k\in\mathbb{N}$. Then condition (A0) is satisfied and conditions (A1) and (A2) are equivalent to
\begin{equation}\label{A1-A2}
\begin{cases}
\mu^k_2 (0, -1) 
+\mu^k_3\left(\frac{3}{k^2}, 1\right)
+\left(-2, -\frac{5}{2}\right)\to (0,0),
\\
\frac{1}{k^3}\mu^k_3\to 0,
\end{cases}
\end{equation}
as $k\to\infty$. Thus, by letting $\mu^k_2=\frac{2k^2}{3}-\frac{5}{2}$ and $\mu^k_3=\frac{2k^2}{3}$ for all $k\geq 2$, we can check that  condition \eqref{A1-A2} holds. This means that sequences $\{x^k\}$ and  $\{\mu^k\}$ satisfy  conditions (A0)--(A2), as required. 
}
\end{example}

The next result shows that, for convex problems of the form \eqref{problem}, the AKKT condition is not only a necessary optimality condition but also a sufficient  one. Recall that problem \eqref{problem} is called {\em  convex} if and only if the functions $f_i, i=1, \ldots, m,$ are convex  and the mapping $g$ is $\Theta$-convex, i.e.,
\begin{equation*}
g(tx + (1 - t)y) - tg(x) - (1 - t)g(y) \in -\Theta,
\end{equation*}  
for all $x, y\in \mathbb{R}^n$ and $t\in [0, 1]$.

\begin{theorem} \label{Sufficient-AKKT-Theorem}
Assume that problem \eqref{problem} is convex and $\bar x\in\mathcal{F}$. Then, $\bar x$ is a weak efficient solution of \eqref{problem} if and only if $\bar x$ satisfies the AKKT condition.
\end{theorem}
\begin{proof} Thanks to Theorem \ref{AKKT-Theorem}, we only need to prove the ``if'' part. Arguing by contradiction, assume that $\bar x$ satisfies the AKKT condition but $\bar x$ is not a weak efficient solution of \eqref{problem}. Hence  there exists $\hat{x}\in\mathcal{F}$ such that
\begin{equation}\label{not-efficient}
f_i(\hat{x})<f_i(\bar x), \ \ \forall i=1, \ldots, m.
\end{equation}
By the AKKT condition, there exist $\lambda\in\mathbb{R}^m_+$ with $\sum_{i=1}^{m}\lambda_i=1$, $\{x^k\}\subset \mathbb{R}^n$ and $\{\mu^k\}\subset \Theta_+$ satisfying conditions (A0)--(A2). Since $f_i$, $i=1,\ldots, m$, are convex functions, we have
\begin{equation}\label{fi-convex}
f_i(\hat{x})\geq f_i(\bar x)+ \langle \nabla f(\bar x), \hat{x}-\bar{x}\rangle.
\end{equation}
By the $\Theta$-convexity of $g$, it is easily seen that $\langle \mu^k, g(\cdot)\rangle$ is a convex function. Hence 
\begin{equation}\label{mk-g-convex}
\langle \mu^k, g(\hat{x})\rangle\geq \langle \mu^k, g(x^k)\rangle+\langle \nabla g(x^k)^\ast\mu^k, \hat{x}-x^k\rangle,  \ \ \forall k\in\mathbb{N}. 
\end{equation}
Multiplying \eqref{fi-convex} by $\lambda_i$ and adding up then gives, from \eqref{mk-g-convex} and the facts that $\hat{x}\in\mathcal{F}$  and $\mu^k\in\Theta_+$, that
\begin{align}
\sum_{i=1}^{m}\lambda_i f_i(\hat{x})&\geq \sum_{i=1}^{m}\lambda_i f_i(\hat{x})+\langle \mu^k, g(\hat{x})\rangle\notag
\\
&\geq\sum_{i=1}^{m}\lambda_i (f_i(\bar x)+ \langle \nabla f(\bar x), \hat{x}-\bar{x}\rangle)+\langle \mu^k, g(x^k)\rangle+\langle \nabla g(x^k)^\ast\mu^k, \hat{x}-x^k\rangle\notag
\\
&=\sum_{i=1}^{m}\lambda_i f_i(\bar x)+\sigma_k \label{sum-k}
\end{align}
for all $k\in\mathbb{N}$, where 
$$\sigma_k:=\langle \mu^k, g(x^k)\rangle+\sum_{i=1}^{m}\lambda_i\langle \nabla f(\bar x), \hat{x}-\bar{x}\rangle+\langle \nabla g(x^k)^\ast\mu^k, \hat{x}-x^k\rangle.$$
We claim that $\sigma_k\to 0$ as $k$ tends to infinity. Indeed, for each $k\in\mathbb{N}$, we have
\begin{equation*}
\sigma_k=\langle \mu^k, g(x^k)\rangle+\left\langle \nabla g(x^k)^\ast\mu^k+ \sum_{i=1}^{m}\lambda_i\nabla f_i(\bar x), \hat{x}-\bar{x}\right\rangle-\langle\nabla g(x^k)^\ast\mu^k, x^k-\bar{x}\rangle.
\end{equation*}
By conditions (A1) and (A2), the two first terms of $\sigma_k$ converge to $0$ as $k\to\infty$. Moreover, from condition (A1) it follows that the sequence $\{\nabla g(x^k)^\ast\mu^k\}$ is bounded. This and condition (A0) imply that  the last term of $\sigma_k$ also tends to $0$. Hence  $\sigma_k\to 0$ as $k$ tends to infinity, as required.

Now, taking $k\to\infty$ in \eqref{sum-k} yields 
\begin{equation}\label{sum-ge}
\sum_{i=1}^{m}\lambda_i f_i(\hat{x})\geq \sum_{i=1}^{m}\lambda_i f_i(\bar x).
\end{equation}
From $\sum_{i=1}^{m}\lambda_i=1$ and  \eqref{not-efficient} it follows that
\begin{equation*}
\sum_{i=1}^{m}\lambda_i f_i(\hat{x})< \sum_{i=1}^{m}\lambda_i f_i(\bar x),
\end{equation*}
contrary to \eqref{sum-ge}. The proof is complete.
\end{proof}

We now introduce some strict constraint qualifications of \eqref{problem}.   Recall that a property P is called a {\em strict constraint qualification}  (SCQ) of  \eqref{problem} if and only if the implication
\begin{equation}\label{strict-CQ}
\mathrm{AKKT+P}\ \ \Rightarrow\ \ \mathrm{KKT} 
\end{equation}
is true; see \cite[Definition 3.3]{Birgin2014}. SCQs are important because they are sufficient conditions to guarantee that limits of AKKT sequences are KKT points. Furthermore, since every local weak efficient solution  satisfies the AKKT condition, the property \eqref{strict-CQ} shows that every SCQ is also a constraint qualification of \eqref{problem}. 

The first SCQ of \eqref{problem} is as follows.
\begin{definition}
{\rm 
We say that $\bar x\in\mathcal{F}$ satisfies the {\em bounded approximate Karush--Kuhn--Tucker  condition}  (BAKKT) if and only if there exist a vector $\lambda\in \mathbb{R}^m_+$ with $\sum_{i=1}^{m}\lambda_i=1$ and sequences $\{x^k\}\subset \mathbb{R}^n$ and $\{\mu^k\}\subset \Theta_+$ such that conditions (A0)--(A2) hold and the sequence $\{\mu^k\}$ is bounded.   
}
\end{definition}

\begin{theorem}\label{BAKKT-Theorem}  If  $\bar x\in\mathcal{F}$ satisfies the BAKKT condition, then so does the KKT condition. (That is, the BAKKT condition is an SCQ.)
\end{theorem}
\begin{proof} By assumption, there exist $\lambda\in\mathbb{R}^m_+$ with $\sum_{i=1}^{m}\lambda_i=1$, $\{x^k\}\subset \mathbb{R}^n$ and $\{\mu^k\}\subset \Theta_+$ such that $x^k\to \bar x$ and
\begin{align}
&\nabla g(x^k)^* \mu^k+\sum_{i=1}^m\lambda_i\nabla f_i(\bar x)\to 0,\label{BAKKT-1}
\\
&\langle \mu^k, g(x^k)\rangle\to 0,\label{BAKKT-2}
\end{align} 
and the sequence $\{\mu^k\}$ is bounded. Thanks to \cite[Theorem 3.16]{Brezis}, without  loss of generality, we may assume that $\mu^k\stackrel{w^\ast}\longrightarrow \mu$ as $k\to\infty$. Since $\Theta_+$ is weak$^\ast$ closed, one has $\mu\in\Theta_+$. Taking the limit in \eqref{BAKKT-1} and \eqref{BAKKT-2} as $k\to\infty$, we obtain 
$$\nabla g(\bar x)^*\mu+\sum_{i=1}^m\lambda_i\nabla f_i(\bar x)=0\ \ \text{and} \ \ \langle \mu, g(\bar x)\rangle=0.$$ 
This means that $\bar x$ satisfies the KKT condition.  The proof is complete. 
\end{proof}
We next show that if the cone $\Theta$ is dually compact, then the Robinson constraint qualification is a sufficient condition for the BAKKT condition.
\begin{definition}[{see \cite{Robinson}}]
{\rm We say that $\bar x\in\mathcal{F}$ satisfies the {\em  Robinson constraint qualification} (RCQ) if and only if  
$$0\in \mathrm{int}\left[g(\bar x)+\nabla g(\bar x)(\mathbb{R}^n)+\Theta\right].$$  
}
\end{definition}  
\begin{remark}
{\rm Thanks to \cite[Proposition 2.95]{Bonnans}, the RCQ is equivalent to the following condition
\begin{equation}\label{RCQ}
\nabla g(\bar x)(\mathbb{R}^n)+\mathrm{cone}\,(\Theta+g(\bar{x}))=Y. 
\end{equation}
Moreover, if $\Theta$ has a nonempty interior, then by \cite[Lemma 2.99]{Bonnans} the RCQ is equivalent to the Mangasarian--Fromovitz constraint qualification:
$$\exists d\in \mathbb{R}^n\;:\;g(\bar x)+\nabla g(\bar x)d\in -\mathrm{int}\,\Theta.$$ 
}
\end{remark}

\begin{definition}[{see \cite[Definition 3.1]{Ng-2005}}]   
{\rm 
We say that a closed convex cone $\Theta\subset Y$ is {\em  dually compact} if and only if there exists a compact subset $K$ of $Y$ such that 
\begin{equation*}
\Theta_+\subset \mathcal{W}(K),
\end{equation*}
where
\begin{equation*}
\mathcal{W}(K):=\left\{y^\ast\in Y^\ast\;:\; \|y^\ast\|\leq \sup\{\langle y^\ast, y\rangle\;:\; y\in K\}\right\}.
\end{equation*}
}
\end{definition}
The following remark summarizes some important facts of the dually compactness.  For more information, the readers are referred to \cite{Ng-2005,Durea2011} and the references therein. 
\begin{remark}\label{Remark-dually-compact}
{\rm 
\begin{enumerate}[(i)]
\item The cone $\Theta$ is dually compact if and only if there exists a finite subset $P:=\{y_1, y_2, \ldots, y_p\}\subset Y$ such that
\begin{equation*}\label{finte-dually-compact}
\Theta_+\subset  \left\{y^\ast\in Y^\ast\;:\; \|y^\ast\|\leq \max\{\langle y^\ast, y_i\rangle\;:\; y_i\in P\}\right\}.
\end{equation*}

\item The cone $\Theta$ is dually compact if and only if $\Theta_+$ is weak$^\ast$ locally compact.

\item If $\mathrm{dim}\,Y<\infty$ or $\mathrm{int}\,\Theta\not=\emptyset$, then $\Theta$ is dually compact. The converse does not hold in general. For example,  the  cone
\begin{equation*}
\Theta:=\left\{(0, \theta_1, \theta_2, \ldots)\in l_2 \;:\; \sum_{k=2}^{\infty}\theta_k^2\leq +\infty\right\}
\end{equation*}
has an empty interior, but it is dually compact. 

\item If $\Theta$ is dually compact, then
\begin{equation*}\label{dually-compact}
\mu^k\stackrel{w^\ast}\longrightarrow 0 \Leftrightarrow \mu^k\to 0 \ \ \text{for any (generalized) sequence}\ \ \{\mu^k\}\subset \Theta_+.
\end{equation*}
\end{enumerate}
}
\end{remark}
 
\begin{theorem}\label{MFCQ-AKKT} Assume that $\Theta$ is dually compact.  Let $\bar x\in\mathcal{F}$ be such that the AKKT condition and the RCQ are satisfied.  Then the BKKT condition holds and so does the KKT condition. 
\end{theorem}
\begin{proof} Since $\bar x$ satisfies the AKKT condition, there exist $\lambda\in\mathbb{R}^m_+$ with $\sum_{i=1}^{m}\lambda_i=1$, $\{x^k\}\subset \mathbb{R}^n$ and $\{\mu^k\}\subset \Theta_+$ such that $x^k\to \bar x$ and
\begin{align}
&\nabla g(x^k)^* \mu^k+\sum_{i=1}^m\lambda_i\nabla f_i(\bar x)\to 0,\label{MFCQ-1}
\\
&\langle \mu^k, g(x^k)\rangle\to 0.\label{MFCQ-2}
\end{align} 
We claim that the sequence $\{\mu^k\}$ is bounded. Indeed, if otherwise, then we may assume without  loss of generality that $\|\mu^k\|\to \infty$ as $k\to\infty$. For each $k\in\mathbb{N}$, put $\tilde{\mu}^k:=\frac{\mu^k}{\|\mu^k\|}$. Then, $\|\tilde{\mu}^k\|=1$ for all $k\in\mathbb{N}$.  Selecting a subsequence if necessary we may assume that $\{\tilde{\mu}^k\}$ converges  weakly$^\ast$ to some $\tilde{\mu}\in\Theta_+$. By the dually compactness, we see that  $\tilde\mu$ is not null. Indeed, if otherwise, then $\mu^k\to 0$ in norm due to Remark \ref{Remark-dually-compact}(iv)  and this is not possible since $\|\mu^k\|=1$ for all $k\in\mathbb{N}$.   From \eqref{MFCQ-1} and \eqref{MFCQ-2}, we have
\begin{align*}
&\nabla g(\bar x)^\ast\tilde{\mu}=\lim\limits_{k\to\infty}\left[\nabla g(x^k)^* \tilde{\mu}^k+\frac{1}{\|\mu^k\|}\sum_{i=1}^m\lambda_i\nabla f_i(\bar x)\right]=0,  
\\
&\langle \tilde{\mu}, g(\bar x)\rangle=\lim\limits_{k\to\infty}\langle \tilde{\mu}^k, g(x^k)\rangle=0. 
\end{align*}
Therefore 
\begin{align*}
\langle \tilde{\mu}, \nabla g(\bar{x})d+t(\theta+g(\bar{x}))\rangle&=\langle\nabla g(\bar x)^*\tilde{\mu}, d\rangle+t\langle\tilde{\mu}, g(\bar x)\rangle+t\langle\tilde{\mu}, \theta\rangle
\\
&=t\langle\tilde{\mu}, \theta\rangle\geq 0,
\end{align*}
for all $d\in\mathbb{R}^n$, $\theta\in\Theta$ and $t\geq 0$. This means that
\begin{equation*}
\langle \tilde{\mu}, \nabla g(\bar{x})d+v\rangle\geq 0, \ \ \forall d\in\mathbb{R}^n, v\in \mathrm{cone}\,(\Theta+g(\bar x)).
\end{equation*}
Thanks to \eqref{RCQ}, we obtain
\begin{equation*}
\langle\tilde{\mu}, y\rangle\geq 0, \ \ \forall y\in Y.
\end{equation*}
This implies that $\tilde{\mu}=0$, a contradiction. Hence   the sequence $\{\mu^k\}$ is bounded. The poof is complete.  
\end{proof}
The following example shows that the BAKKT condition does not imply the RCQ condition.
\begin{example}
{\rm Consider the following problem:
\begin{align*}
	& \text{min}\, f(x) 
	\\
	&\text{s.t.}\ \ x\in \mathcal{F}:=\{x\in\mathbb{R}^2\,:\, g(x)\in -\Theta\},
\end{align*} 
where $\Theta=\mathbb{R}^3_+$ and
\begin{equation*}
f(x)=x_1, g(x)=(g_1(x), g_2(x), g_3(x)), g_1(x)=-x_1, g_2(x)= -x_2, g_3(x)= x_2,  
\end{equation*}
for all $x=(x_1, x_2)\in\mathbb{R}^2$. Then we have that the feasible set is  $\mathcal{F}=\mathbb{R}_+\times\{0\}$ and $\bar x=(0,0)$ is a minimum point of $f$ on $\mathcal{F}$. Since
\begin{equation*}
\nabla f(\bar x)=(1, 0), \nabla g_1(\bar x)=(-1, 0), \nabla g_2(\bar x)=(0, -1), \nabla g_3(\bar x)=(0, 1),
\end{equation*}
it is easy to check that $\bar x$ satisfies the KKT (BAKKT)  but not   the RCQ condition.
}
\end{example}
We are now introducing the weakest SCQ associated with the AKKT condition called the AKKT-regularity. Let $K(\cdot, \cdot)\colon \mathbb{R}^n\times\mathbb{R}_+\rightrightarrows \mathbb{R}^n$ be a set-valued mapping defined by
$$K(x, r):=\bigg\{\nabla g(x)^*\mu \;:\; \mu\in\Theta_+, |\langle\mu, g(x)\rangle|\leq r\bigg\},\ \ \ \forall x\in\mathbb{R}^n, r\in\mathbb{R}_+.$$
The set  $K(x, r)$ is as a perturbation of $K(\bar x, 0)$ around a given point $(\bar x, 0)$. This set is always nonempty and convex. Moreover, we have $K(x, \alpha r)=\alpha K(x, r)$ for all $\alpha>0$. 
\begin{definition}
{\rm We say that the {\em  AKKT-regularity} holds at $\bar x\in\mathcal{F}$ if and only if the set-valued mapping $K(\cdot, \cdot)$ is outer semicontinuous at $(\bar x, 0)$, that is,  
\begin{equation}\label{AKKT-regular}
\Limsup_{(x, r)\to (\bar x, 0)} K(x, r)\subset K(\bar x, 0),
\end{equation}
where
\begin{equation*}
\Limsup_{(x, r)\to (\bar x, 0)} K(x, r):=\{\bar w:\exists\, (x^k,  r_k, w^k)\to (\bar x, 0, \bar w)\ \ \text{with}\ \  w^k\in K(x^k, r_k),\,\forall k\in\mathbb{N}\}.
\end{equation*}
}
\end{definition}

The following theorem shows that the AKKT-regularity is the weakest  SCQ.
\begin{theorem}
A feasible point $\bar x$ is the AKKT-regularity if and only if for every continuously differentiable objective function $f$ in \eqref{problem} such that the AKKT condition holds at $\bar x$, we have that the KKT  condition also holds at $\bar x$. 
\end{theorem}
\begin{proof}
$(\Rightarrow):$ Assume that the AKKT-regularity holds at $\bar x\in\mathcal{F}$
and $f$ is an arbitrary continuously differentiable objective function such that the AKKT holds at $\bar x$. Then, there exist $\lambda\in\mathbb{R}^m_+$ with $\sum_{i=1}^{m}\lambda_i=1$, $\{x^k\}\subset \mathbb{R}^n$ and $\{\mu^k\}\subset \Theta_+$ such that $x^k\to \bar x$ and
\begin{align}
&\nabla g(x^k)^* \mu^k\to-\sum_{i=1}^m\lambda_i\nabla f_i(\bar x),\label{AKKT-regular-1}
 \\
&\langle \mu^k, g(x^k)\rangle\to 0.\label{AKKT-regular-2}
\end{align}
For each $k\in\mathbb{N}$, put $r_k:=|\langle \mu^k, g(x^k)\rangle| \ \ \text{and}\ \ w^k:=\nabla g(x^k)^* \mu^k.$  
Then, we have $w^k\in K(x^k, r_k)$ for all $k\in\mathbb{N}$. From \eqref{AKKT-regular-1}, \eqref{AKKT-regular-2} and the AKKT-regularity it follows that  
\begin{equation*}
-\sum_{i=1}^m\lambda_i\nabla f_i(\bar x)\in K(\bar x, 0).
\end{equation*}
Hence  $\bar x$ satisfies the KKT condition.

$(\Leftarrow):$   Let $\bar x\in \mathcal{F}$. Assume that for any continuously differentiable objective function $f$ in \eqref{problem} such that if the AKKT condition holds at  $\bar x$, then the KKT  condition also holds at this point. We claim that  $\bar x$ satisfies the AKKT-regularity, i.e., the inclusion \eqref{AKKT-regular} holds. Indeed, let $\bar w$ be an arbitrary element in the left-hand side of \eqref{AKKT-regular}. Then, there exists a sequence $\{(x^k,  r_k, w^k)\}$ converging to $(\bar x, 0, \bar w)$ such that $w^k\in K(x^k, r_k)$   for all $k\in\mathbb{N}$. Hence, for each $k\in\mathbb{N}$, there exists $\mu^k\in\Theta_+$ such that
\begin{equation*}
w^k=\nabla g(x^k)^*\mu^k\ \ \text{and}\ \   |\langle \mu^k, g(x^k)\rangle|\leq r_k.
\end{equation*} 
Let $f$ be the linear function defined by $f(x)=-\langle \bar w, x\rangle$ for all $x\in\mathbb{R}^n$. Since $w^k\to \bar w$ and $\nabla f(\bar x)=-\bar w$, one has 
$$\nabla g(x^k)^*\mu^k+\nabla f(\bar x)\to 0.$$ 
Moreover, since $r_k\to 0$, we have  $\langle \mu^k, g(x^k)\rangle\to 0.$ 
By assumption, the KKT condition holds at $\bar x$, i.e., there exists $\mu\in\Theta_+$ such that
\begin{equation*}
\nabla g(\bar x)^*\mu+\nabla f(\bar x)=0 \ \ \text{and}\ \ \langle \mu, g(x)\rangle=0.
\end{equation*}
Hence  
$$\bar w=-\nabla f(\bar x)\in K(\bar x, 0),$$ 
as required. 
\end{proof}

We finish this section by presenting an example to show that although the AKKT-regularity is the weakest constraint qualification associated with the AKKT condition, this condition alone does not imply the RCQ one.
\begin{example}
{\rm Let the feasible set $\mathcal{F}$ of \eqref{problem} be defined by
\begin{equation*}
\mathcal{F}:=\{x\in\mathbb{R}^2\;:\; g(x)\in -\Theta\},
\end{equation*}
where
\begin{equation*}
g(x)=(g_1(x), g_2(x)), g_1(x)=x_1, g_2(x)=x_1^2, \ \ \forall x=(x_1, x_2)\in\mathbb{R}^2,
\end{equation*}
and $\Theta=\{0\}\times\mathbb{R}_+$. Clearly, $\bar x=(0, 0)\in \mathcal{F}$. By direct calculations, we get
\begin{equation}\label{not-RCQ}
\nabla g_1(x)=(1, 0), \nabla g_2(x)=(2x_1, 0), \ \ \forall x=(x_1, x_2)\in\mathbb{R}^2,
\end{equation}  
and $\Theta_+=\mathbb{R}\times\mathbb{R}_+$. Hence 
\begin{align*}
g(\bar x)+\nabla g(\bar x)(\mathbb{R}^n)+\Theta=\mathbb{R}\times \mathbb{R}_+.
\end{align*}
This implies that $(0,0)$ is not an interior point of $g(\bar x)+\nabla g(\bar x)(\mathbb{R}^n)+\Theta$, i.e., the RCQ does not hold at $\bar x$.

Now let us prove that the AKKT-regularity holds at $\bar x$. Since \eqref{not-RCQ}, we have
\begin{align*}
K(x, r)&=\{(\mu_1+2x_1\mu_2, 0)\,:\, \mu_1\in\mathbb{R}, \mu_2\in \mathbb{R}_+, |\mu_1x_1+\mu_2x_1^3|\leq r\},\;\forall x\in\mathbb{R}^2, r\in\mathbb{R}_+,
\\
K(\bar x, 0)&=\{(\mu_1, 0)\,:\, \mu_1\in\mathbb{R}\}=\mathbb{R}\times\{0\}.
\end{align*}
Clearly, $K(x, r)\subset K(\bar x, 0)$ for all $x\in\mathbb{R}^2$ and  $r\in\mathbb{R}_+$. Thus the AKKT-regularity holds at $\bar x$.
}
\end{example}

\end{document}